\documentclass[11pt]{article}
\title{Small-dimensional representations of algebraic groups of type $A_l$}

\usepackage[numbers]{natbib}
\usepackage[hyphens]{url}
\usepackage{breakurl}
\usepackage{hyperref}
\usepackage{booktabs}
\usepackage[british]{babel}

\usepackage{amsmath}
\usepackage{amsfonts}
\usepackage{amssymb}

\usepackage[utf8]{inputenc}
\usepackage[parfill]{parskip}
\usepackage[margin=1in]{geometry}
\usepackage{setspace}
\usepackage{enumitem}
\usepackage{tabularx}
\usepackage{float}
\usepackage{mathrsfs}

\usepackage{amsthm}
\usepackage[many]{tcolorbox}

\newtheorem{theorem}{Theorem}[section]
\newtheorem{lemma}{Lemma}[section]

\theoremstyle{definition}

\begin{document}
\author{\'Alvaro L. Mart\'inez}
\date{\today}
\maketitle

\begin{abstract}
\noindent For $G$ an algebraic group of type $A_l$ over an algebraically closed field of characteristic $p$, we determine all irreducible rational representations of $G$ in defining characteristic with dimensions $\le (l+1)^s$ for $s=3,4$, provided that $l>18$, $l>35$ respectively. We also give explicit descriptions of the corresponding modules for $s=3$.
\end{abstract}
\thispagestyle{empty}


\pagenumbering{arabic}

\section{Introduction}

Let $K$ be an  algebraically closed field of characteristic $p>0$. For every simply connected simple linear algebraic group over $K$ of rank $l$, the irreducible rational representations in defining characteristic with dimension below a bound proportional to $l^2$ were determined by Liebeck in \cite{liebeck}. Lübeck \cite{lubeck} extended these results taking a bound proportional to $l^3$. For groups of type $B_l$, $C_l$ and $D_l$ this bound was $l^3$, but for type $A_l$ the bound taken was $l^3/8$. This is consistent with the fact that the dimension of the natural module for the group is roughly twice the rank for $B_l$, $C_l$ and $D_l$, but merely $l+1$ for type $A_l$. However, a larger bound for the latter is desirable for some applications (see for example \cite{halasi}, \cite{lee}).

Let $G=\mathrm{SL}_{l+1}(K)$. As explained in §\ref{prelim}, the irreducible $KG$-modules can be parameterised by their highest weights, and we denote them by $L(\lambda)$. Their dimensions are readily obtained once the weight multiplicities of the subdominant weights are known. For small ranks, Lübeck has developed software and provided lists of weight multiplicities and dimensions of highest weight modules \cite{lubpage}. For type $A_l$, the lists include all modules with $\dim L(\lambda) \le (l+1)^4$ and ranks $l\le 20$. In this note we first determine for arbitrary rank $l$ the highest weights and dimensions of the irreducible modules of groups of type $A_l$ with dimensions $\le (l+1)^3$, as summarised in the following theorem. Throughout the paper, $\epsilon_p(k)$ will denote 1 if $p$ divides $k$ and $0$ otherwise.\\

\begin{theorem} \label{thmthird}

Let $G=\mathrm{SL}_{l+1}(K)$ and $l>18$. Table \ref{tablethird} contains all nonzero $p$-restricted dominant weights $\lambda$ up to duals such that $\dim L(\lambda)\le(l+1)^3$, as well as the dimensions of the corresponding modules $L(\lambda)$.

\end{theorem}

We further extend these techniques, mainly applying results of Seitz \cite{seitz} and Cavallin \cite{cavallin}, for modules with $\dim L(\lambda) \le (l+1)^4$, proving the following.\\

\begin{theorem} \label{thmfourth}

Let $G=\mathrm{SL}_{l+1}(K)$ and $l>35$. Tables \ref{tablethird} and \ref{tablefourth} contain all nonzero $p$-restricted dominant weights $\lambda$ up to duals such that $\dim L(\lambda)\le(l+1)^4$, as well as the dimensions of the corresponding modules $L(\lambda)$.

\end{theorem}

\begin{table}
  \setlength{\tabcolsep}{3em}
  \renewcommand{\arraystretch}{1.4}
  \begin{tabular}{ l l l } 
    \specialrule{.1em}{.05em}{.05em} 
    $\lambda$ & $\dim L(\lambda)$ & Conditions\\ 
    \specialrule{.1em}{.05em}{.05em} 
    $\lambda_1$ & $l+1$& \\ 
    $\lambda_2$& $\binom{l+1}{2}$ &\\ 
    $2\lambda_1$ & $\binom{l+2}{2}$& \\ 
    $\lambda_1+\lambda_l$ & $(l+1)^2-1-\epsilon_p(l+1)$& \\ 
    $\lambda_3$ & $\binom{l+1}{3}$ &\\ 
    $3\lambda_1$ & $\binom{l+3}{3}$&  \\ 
    $\lambda_1+\lambda_2$ & $2\binom{l+2}{3}- \epsilon_p(3)\binom{l+1}{3}$& \\ 
    $\lambda_1+\lambda_{l-1}$& $3\binom{l+2}{3}-\binom{l+2}{2}-\epsilon_p(l)(l+1)$ &\\ 
    $2\lambda_1+\lambda_{l}$& $3\binom{l+2}{3}+\binom{l+1}{2}-\epsilon_p(l+2)(l+1)$ & \\ 
    \hline
    $\lambda_4$ & $\binom{l+1}{4}$ & $l\le 28$\\ 
    \specialrule{.1em}{.05em}{.05em} 
  \end{tabular}
  \caption{Nonzero $p$-restricted dominant weights $\lambda$ such that $\dim L(\lambda)\le (l+1)^3$ for $l>18$.}
  \label{tablethird}

\bigskip

  \setlength{\tabcolsep}{1.4em}
  \renewcommand{\arraystretch}{1.4}
  \begin{tabular}{ l l l } 
    \specialrule{.1em}{.05em}{.05em} 
    $\lambda$ & $\dim L(\lambda)$ & Conditions\\ 
    \specialrule{.1em}{.05em}{.05em} 
    $\lambda_4$& $\binom{l+1}{4}$& \\ 
    $4\lambda_1$ & $\binom{l+4}{4}$&  \\ 
    $2\lambda_2$ & $\binom{l+1}{2}^2-(l+1)\binom{l+1}{3}-\epsilon_p(3)\binom{l+1}{4}$ & \\ 
    $\lambda_1+\lambda_3$ & $3\binom{l+2}{4}-\epsilon_p(2)\binom{l+1}{4} $& \\
    $2\lambda_1+\lambda_2$ & $3\binom{l+3}{4}$&\\ 
    $\lambda_1+\lambda_{l-2}$ & $(l-2)\binom{l+2}{3}-\epsilon_p(l-1)\binom{l+1}{2} $ & \\
    
    $3\lambda_1+\lambda_l$ & $4\binom{l+3}{4}+\binom{l+2}{3}-\epsilon_p(l+3)\binom{l+2}{2}$& \\ 
    
    $2\lambda_1+\lambda_{l-1}$ & $\binom{l+3}{2} \binom{l}{2}-\epsilon_p(l+1)((l+1)^2 -2)$ & \\ 
    
    $\lambda_2+\lambda_{l-1}$ & $\binom{l+1}{2}^2 - (l+1)^2 - \epsilon_p(l-1)((l+1)^2-1)-\epsilon_p(l)$ & \\ 
    
    $2\lambda_1+2\lambda_l$ & $\binom{l+2}{2}^2-(l+1)^2-\epsilon_p(l+3)((l+1)^2-1)-\epsilon_p(l+2)$ &\\ 
    
    $\lambda_1+\lambda_2+\lambda_{l}$ & $(l+1)(2\binom{l+1}{3}+l^2-1)-4 \epsilon_p(3) (l-2)(\binom{l+1}{3}-1) $ &\\
     & $-\epsilon_p(l) \binom{l+2}{2}-\epsilon_p(l+2)(1-\epsilon_p(3))\binom{l+1}{2}$ &\\
    \hline
    $\lambda_5$& $\binom{l+1}{5}$& $l\le 128$ \\ 

    $\lambda_2+\lambda_{3}$& $\binom{l+1}{2} \binom{l+1}{3} -(l+1) \binom{l+1}{4}-\epsilon_p(2)\binom{l+1}{5}-4\epsilon_p(3)\binom{l+2}{5}$ &$l\le 109$\\ 
    $5\lambda_1$& $\binom{l+5}{5}$& $l\le 108$ \\ 
    $3\lambda_1+\lambda_2$ & $4\binom{l+4}{5}-\epsilon_p(5)(3\binom{l+3}{5}+2\binom{l+2}{4}+\binom{l+1}{3})$&$l\le 108$\\
        $\lambda_1+\lambda_{4}$& $4 \binom{l+2}{5}-\epsilon_p(5)\binom{l+1}{5}$ &$l\le 42$\\

    \specialrule{.1em}{.05em}{.05em} 
  \end{tabular}
  \caption{Nonzero $p$-restricted dominant weights $\lambda$ not in Table \ref{tablethird} such that $\dim L(\lambda) \le (l+1)^4$ for $l>35$.}
  \label{tablefourth}
\end{table}

\emph{Remark.} In Theorem \ref{thmfourth}, if we relax the condition on $l$ to $l>20$, then the values of $\lambda$ that need to be added to Table \ref{tablefourth} are:
\begin{description}
	\item $2\lambda_1+\lambda_3$, for $l \le 35$, with $\dim L(\lambda)=6 \binom{l+3}{5}-\epsilon_p(5)(3 \binom{l+2}{5}+\binom{l+1}{4})$,
    \item $\lambda_6$ for $l \le 32$, with $\dim L(\lambda)=\binom{l+1}{6}$,
	\item $\lambda_1+\lambda_{l-3}$, for $l\le 28$, with $\dim L(\lambda)=(l-3)\binom{l+2}{4}-\epsilon_p(l-2) \binom{l+1}{3}$,
    \item $\lambda_7$, for $l\le 22$, with $\dim L(\lambda)=\binom{l+1}{7}$.
\end{description}
 This can be shown via a lengthier variant of the proof of Theorem \ref{thmfourth}.
 
The proofs of Theorems \ref{thmthird} and \ref{thmfourth} appear in sections \ref{firstproof} and \ref{secondproof}. Results similar to Theorem \ref{thmfourth} for types $B_l$, $C_l$ and $D_l$ will appear in forthcoming work.

In §\ref{descriptions}, we also provide explicit descriptions of the modules with dimensions $\le (l+1)^3$, as quotients of subspaces of the tensor product $V^{\otimes k}$, by combining the Young symmetrizers construction in \cite{fultonharris} and a result of Cavallin (\cite{cavallin}, Lemma 4.1.2).

\section{Preliminaries} \label{prelim}

Let $G=\mathrm{SL}_{l+1}(K)$ as in the introduction. Let $T<G$ be a maximal torus of $G$ and $B=UT$ a Borel subgroup of $G$, where $U$ denotes the unipotent radical. Let $X(T)=\mathrm{Hom}(T,K^*)\cong \mathbb{Z}^l$ be the character group of $T$ and fix a set of simple roots $\Pi=\{\alpha_1,...,\alpha_l\} \subset X(T)$, a base of the root system $\Phi$ of $G$. Denote by $\Phi^+$ the set of positive roots. Also let $\{\lambda_1,...,\lambda_l\}$ be the set of fundamental dominant weights corresponding to $\Pi$. Define the partial order $\preccurlyeq$ as follows: for $\lambda, \mu \in X(T)$, $\mu \preccurlyeq \lambda$ if and only if $\lambda-\mu$ is a non-negative linear combination of the simple roots.

Let $L$ be a finite dimensional $KG$-module. For $\mu \in X(T)$, let $L_{\mu}=\{v \in M : tv=\mu(t)v\ \forall t \in T\}$. If $L_{\mu}\neq \{0\}$, we say that $\mu$ is a weight of $L$ and that $L_{\mu}$ is its corresponding weight space. The module $L$ is the direct sum of its weight spaces $L_{\mu}$. The Weyl group $W=N_G(T)/T$, which for type $A_l$ is isomorphic to the symmetric group $S_{l+1}$, acts on the set of weights. In every $W$-orbit there is exactly one dominant weight, that is, a weight which is a non-negative linear combination of fundamental weights. Every irreducible $KG$-module has a unique maximum weight $\lambda$ with respect to the partial order $\preccurlyeq$, called its highest weight, which is in turn dominant. Its weight space is the unique 1-space fixed by the Borel subgroup $B$. Conversely, for every dominant weight $\lambda$, there exists a unique irreducible $KG$-module (up to isomorphism) with highest weight $\lambda$, and every irreducible $KG$-module arises in this way. We thus parameterise the irreducible $KG$-modules by their highest weights as $L(\lambda)$. We also denote $m_{\lambda}(\mu)=\dim L(\lambda)_{\mu}$, the multiplicity of $\mu$ in $L$. We will denote by $V(\lambda)$ the Weyl module with highest weight $\lambda$. The module $L(\lambda)$ occurs as a composition factor of $V(\lambda)$ exactly once.

Let $\mu=a_1\lambda_1+...+a_l\lambda_l$ be a dominant weight of $L(\lambda)$. We say that $\mu$ is $p$-restricted if $0 \le a_i < p$ for all $i$. We will only consider $p$-restricted highest weights. The reason for this is Steinberg's tensor product theorem (\cite{steinberg}, §11), which states that all irreducible modules can be obtained as tensor products of twists of modules with $p$-restricted highest weights.

It is a basic fact that if $\mu \in X(T)$ is a non-negative linear combination of fundamental weights such that $\mu \prec \lambda$, then $\mu$ is a weight of $V(\lambda)$. Premet's theorem \cite{premet} implies, for type $A_l$, that all such $\mu$ are weights of $L(\lambda)$ as well. We say that such a weight $\mu$ is subdominant. Since weight spaces corresponding to $W$-conjugate weights have equal dimensions, we have the following equality:

\begin{equation} \label{dimbound}
\dim L(\lambda)=\sum_{\mu \preccurlyeq \lambda} \left\vert{ \mu^W }\right\vert m_\lambda(\mu).
\end{equation}

Premet's theorem then implies $\dim L(\lambda)  \ge \sum_{\mu \preccurlyeq \lambda} \left\vert{ \mu^W }\right\vert$. We also note that $m_{\lambda}(\lambda)=1$.

The size of the orbit of a dominant weight can readily be obtained as follows. Write $\mu=a_1\lambda_1+...+a_l\lambda_l$ and let $i_1<...<i_{N_\mu}$ be the indices in $\{1,...,l\}$ corresponding to the nonzero $a_i$'s. The stabiliser in $W$ of $\mu$ is the parabolic subgroup generated by the reflections along the simple roots $\alpha_i$ such that $a_i=0$. For type $A_l$, this means that that $W_\mu \cong S_{i_1} \times S_{i_2-i_1}\times ...\times S_{i_{N_{\mu}}-i_{N_{\mu}-1}} \times S_{l+1-i_{N_\mu}}$ and therefore

\begin{equation} \label{orbitsize}
\left\vert{ \mu^W }\right\vert=\left\vert{W:W_\mu}\right\vert=\frac{(l+1)!}{i_1!(i_2-i_1)!\cdots(i_{N_{\mu}}-i_{N_{\mu}-1})!(l+1-i_{N_\mu})!}.
\end{equation}

Finally, the weight multiplicities for the Weyl module $V(\lambda)$ can be obtained using Freudenthal's formula (e.g. \cite{fultonharris}, §25.1). For type $A_l$, there is a combinatorial way of finding them (see Young's rule, \cite{james}, chapter 14).

\section{Proof of Theorem \ref{thmthird}} \label{firstproof}

The proof has two parts. In the first part, for each dominant weight $\lambda$ in Table \ref{tablethird}, we determine the dimension of $L(\lambda)$. In the second part, we prove that the stated weights are indeed all the $p$-restricted dominant weights that correspond to representations of dimension $\le(l+1)^3$.

\subsection{Dimensions of the modules} \label{dimsthird}

    The first dominant weights in Table $\ref{tablethird}$ correspond to well-known modules: $\lambda_1$ for the natural module (dimension $l+1$); $\lambda_k$ for the $k$-th exterior power (dimension $\binom{l+1}{k}$); $k\lambda_1$ for the $k$-th symmetric power (dimension $\binom{l+k}{k}$) and $\lambda=\lambda_1+\lambda_l$ for the adjoint module, of dimension $(l+1)^2-1-\epsilon_p(l+1)$.
    
     We are left with the weights $\lambda_1+\lambda_2$, $\lambda_1+\lambda_{l-1}$ and $2\lambda_1+\lambda_l$. We will make extensive use of the following result, which is part of Lemma 8.6 in \cite{seitz}.\\
    
    \begin{lemma}\label{lmseitz}
    Let $\lambda=a_i \lambda_i+a_j \lambda_j$, $i<j$, $p$-restricted as before with $a_i a_j\neq 0$. Suppose $\mu=\lambda-(\alpha_i+...+\alpha_j)$. Then $m_{\lambda}(\mu)=j-i+1-\epsilon_p(a_i+a_j+j-i)$.
    \end{lemma}

    The dimensions stated in Table \ref{tablethird} easily follow from this. We give as an example the weight $\lambda=2\lambda_1+\lambda_l$, as the other cases can be dealt with in a similar fashion. The subdominant weights are $\lambda-\alpha_1$ and $\lambda-(\alpha_1+...+\alpha_l)$. The multiplicity of $\lambda-\alpha_1$ in the Weyl module is 1, as can easily be seen (using Freudenthal's formula or otherwise) and therefore so it is in $L(\lambda)$ (by Premet's theorem). By Lemma \ref{lmseitz} the subdominant weight $\lambda-(\alpha_1+...+\alpha_l)$ has multiplicity $l-\epsilon_p(l+2)$. Using equation (\ref{orbitsize}) for the size of the respective orbits, the RHS of equation (\ref{dimbound}) now yields the stated dimension.
    
\subsection{Dominant weights} \label{proofdom3}
Let $\lambda$ be a $p$-restricted dominant weight such that $\dim L(\lambda)\le (l+1)^3$ and $l>18$. Our aim is to show that $\lambda$ appears in Table \ref{tablethird}.

Write $\lambda=a_1\lambda_1+...+a_l\lambda_l$, $0 \le a_i < p$, and let $I_\lambda=\{i_1,...,i_{N_\lambda}\}$, $i_1<...<i_{N_\lambda}$ be the set of indices in $\{1,...,l\}$ corresponding to the nonzero $a_i$'s. We define
\[\Delta_\lambda=\mathrm{max}\{i_1,i_2-i_1,...,i_{N_\lambda}-i_{N_\lambda-1}, l+1-i_{N_\lambda}\}.\]
Notice first that $a_l\lambda_1+...+a_1\lambda_l$ is the highest weight of the dual representation of $L(\lambda)$. This will allow us to consider only one of two cases, whenever $\lambda$ is not self-dual.

We start by considering the case where $\Delta_\lambda\le l-4$. In order to minimise the RHS of equation (\ref{orbitsize}), we choose $I_\lambda=\{5\}$ (this is valid assuming $l>8$). The size of the orbit is then $\binom{l+1}{5}$, which exceeds $(l+1)^3$ if $l>14$. We therefore discard this case.
	
	Now assume $\Delta_\lambda=l-3$. The minimum value $|\lambda^W|$ can attain occurs when $I_\lambda=\{4\}$, assuming $l>6$. However, $\lambda_4$ appears in Table \ref{tablethird}. Thus we have to check two cases. If $\lambda = a_4\lambda_4$ and $a_4>1$, then $\lambda-\alpha_4$ is subdominant and its $\lambda_5$ coefficient is $1$. By the previous paragraph, $\dim L(\lambda)>(l+1)^3$ for $l>14$. Otherwise if $\lambda \neq a_4\lambda_4$, the minimum value of $|\lambda^W|$ now occurs when $I_\lambda=\{1,4\}$. The size of the orbit of $\lambda$ then exceeds $(l+1)^3$ for $l>10$. 
	
    In the following, we assume $\Delta_\lambda \ge l-2$. The next statements exhaust the remaining possibilities.
    \begin{enumerate}[label=\alph*)]
	\item If $a_3>0$, necessarily $\lambda=\lambda_3$.
	
		\begin{proof} Note that if $a_1>0$, then $\mu=\lambda-(\alpha_1+\alpha_2+\alpha_3)$ is subdominant and
		\[\dim L(\lambda) \ge |\lambda^W|+|\mu^W| \ge \frac{(l+1)l(l-1)}{2}+\binom{l+1}{4} > (l+1)^3 \text{ for }l>18.\]
		Also if $a_2>0$ or $a_3>1$, set respectively $\mu=\lambda-(\alpha_2+\alpha_3)$ or $\mu=\lambda-\alpha_3$. In both cases $\mu$ is subdominant with nonzero $\lambda_4$ coefficient and $\mu \neq \lambda_4$, which we already know yields $|\mu^W| > (l+1)^3$, so we discard these too.\end{proof}

	We may now assume that $\lambda=a_1\lambda_1+a_2\lambda_2+a_l\lambda_l$ and not consider its dual.
    
	\item If $a_2>0$ and $a_1>0$, then $\lambda=\lambda_1+\lambda_2$.
	
	\begin{proof} Observe that $\lambda-(\alpha_1+\alpha_2)$ is a subdominant weight with nonzero coefficient of $\lambda_3$. By case a) we require that it equals $\lambda_3$. This yields $\lambda=\lambda_1+\lambda_2$.\end{proof}
	
	\item If $a_2>0$ and $a_1=0$, then $a_2=1$.

	\begin{proof} If $a_2>1$ then $\lambda-\alpha_2$ is subdominant with nonzero coefficients of $\lambda_1$ and $\lambda_3$. In view of case a), we can discard this case.\end{proof}
	
	\item If $a_2=1$, $a_1=0$ and $a_{l}>0$, then $\lambda=\lambda_2+\lambda_{l}$.
	
	\begin{proof} If $a_{l}>1$, $\mu=\lambda-\alpha_{l}$ is subdominant with nonzero $\lambda_{l-1}$ coefficient.  Thus $\Delta_\mu=l-3$ but $\mu \neq \lambda_4,\lambda_{l-3}$, a contradiction.\end{proof}
	
	Assume now that $\lambda$ has the form $a_1\lambda_1+a_l \lambda_l$ and $a_1 \ge a_l$.
	
    \item If $a_1>2$, then $\lambda=3\lambda_1$.
	
	\begin{proof} Notice that $\mu=\lambda-\alpha_1$ is as in case b). Setting $\mu=\lambda_1+\lambda_2$ yields $\lambda=3\lambda_1$.\end{proof}
    
    \item If $a_1=2$ and $a_l>0$, then $\lambda=2\lambda_1+\lambda_l$.
	
	\begin{proof} Note that $\mu=\lambda-\alpha_1$ is as in case d). Hence we require $\mu=\lambda_2+\lambda_{l}$ and solving for $\lambda$ yields $\lambda=2\lambda_1+\lambda_l$.\end{proof}

	\end{enumerate}
    
    The weights not considered at this point already lie in Table \ref{tablethird}. This completes the proof of Theorem \ref{thmthird}.

  \section{Proof of Theorem \ref{thmfourth}} \label{secondproof}
  As in the previous section, we split the proof into two: first we establish the dimension of $L(\lambda)$ for the dominant weights in Table \ref{tablefourth}; then we prove that these are all the dominant weights to consider.

\subsection{Dimensions of the modules} \label{dimsfourth}
    In addition to Lemma \ref{lmseitz}, we will need the following results on weight multiplicities. Lemma \ref{testerman121} is essentially due to Seitz (see the proof of 6.7 in \cite{seitz}), but we state it as it appears in Lemma 2.3 from \cite{testerman}. Lemmas \ref{cavallinc111}, \ref{cavallin2221}, \ref{cavallin11001}, \ref{cavallin011} and \ref{cavallin3321} are respectively results 2.3.19, 6.1.3, 6.1.10, 7.4.5 and 6.1.4 from \cite{cavallin}. Lemma \ref{cavallin1221} is the same as result 7.5.6 from \cite{cavallin} for $p \neq 2$. In the case $p=2$, the result is still true; for the proof one must add the term $-\nu_p(2) \chi^{\mu}(\mu)$ to the expression $\nu_c^\mu(T_\sigma)$ in the original proof and proceed similarly.\\
    
    \begin{lemma}\label{testerman121}
    If $\lambda=2a_j \lambda_j$, $1<j<l$, $a_j>1$ and $\lambda-\mu=\alpha_{j-1}+2\alpha_j+\alpha_{j+1}$, then $m_\lambda(\mu)=2-\epsilon_p(a_j+1)$.\\
    \end{lemma}
    
    \begin{lemma} \label{cavallinc111}
    If $\lambda=a_i\lambda_i+a_j\lambda_j$, $ i<j$, $a_i a_j \neq 0$ and $\lambda-\mu=c\alpha_i+\alpha_{i+1}+...+\alpha_j$ with $0<2c \le a_i+1$, then $m_\lambda(\mu)=j-i+1-\epsilon_p(a_i+a_j+j-i)$.\\
    \end{lemma}
    
    \begin{lemma} \label{cavallin2221}
    If $\lambda=a_1\lambda_1+\lambda_j$, $1<j<l$, $a_1>1$ and $\lambda-\mu=2\alpha_1+...+2\alpha_j+\alpha_{j+1}$, then $m_\lambda(\mu)=\binom{j+1}{2}-\epsilon_p(a_1+j)(j)$.\\
    \end{lemma} 
   
    \begin{lemma}\label{cavallin11001}
    If $\lambda=a_1\lambda_1+a_2\lambda_2+a_l\lambda_l$, $a_1 a_2 a_l \neq 0$, and $\lambda-\mu=\alpha_1+...+\alpha_l$, then $m_\lambda(\mu)=2(l-1)-\epsilon_p(a_1+a_2+1)(l-2)-\epsilon_p(a_2+a_l+l-2)-\epsilon_p(a_1+a_2+a_l+l-1)+\epsilon_p(a_1+a_2+1)\epsilon_p(a_2+a_l+l-2)$.\\
    \end{lemma}
    
    \begin{lemma}\label{cavallin011}
    If $\lambda=\lambda_2+\lambda_3$, and $\mu=\lambda_5$, then $m_\lambda(\mu)=5-\epsilon_p(2)-4\epsilon_p(3)$.\\
    \end{lemma}
     
    \begin{lemma}\label{cavallin3321}
    If $\lambda=a_1\lambda_1+\lambda_j$, $1<j<l-1$, $a_1 >2$ and $\lambda-\mu=3\alpha_1+...+3\lambda_j+2\lambda_{j+1}+\lambda_j$, then $m_\lambda(\mu)=\binom{j+2}{3}-\epsilon_p(a_1+j)\binom{j+1}{2}$.\\
    \end{lemma}
    
   \begin{lemma}\label{cavallin1221}
    If $\lambda=\lambda_2+\lambda_j$, $2<j<l$, and $\mu=\lambda_{j+2}$ (with the convention $\lambda_{l+1}=0$), then $m_\lambda(\mu)=\binom{j+1}{2}-1-\epsilon_p(j)(j+1)-\epsilon_p(j+1)$.\\
    \end{lemma}
    
      We now state some facts, see e.g. (\cite{lubeck}, §3) for a more detailed explanation. Let $\mathscr{L}$ be the complex Lie algebra having the same type as $G$, with Chevalley basis $\{e_{\alpha}, f_{\alpha}=e_{-\alpha}, h_{\alpha}: \alpha \in \Phi^+\}$. Let $v^{\lambda}$ be a highest weight vector of the Weyl module $V(\lambda)$. The weight space $V(\lambda)_{\mu}$ is spanned by the set
      
      \[\mathscr{W}_{\lambda,\mu}=\left\{\frac{f_{\beta_1}^{s_1}}{s_1!} \cdots \frac{f_{\beta_N}^{s_N}}{s_N!}v^{\lambda} : N, s_i\in \mathbb{Z}_{\ge 0},\ \beta_i \in \Phi^+, \sum_{i=1}^N s_i\beta_i=\lambda-\mu\right\}.\]
      Take $v,w \in \mathscr{W}_{\lambda,\mu}$, that is, $v=\frac{f_{\beta_1}^{s_1}}{s_1!} \cdots \frac{f_{\beta_N}^{s_N}}{s_N!}$, $w=\frac{f_{\beta_1}^{t_1}}{t_1!} \cdots \frac{f_{\beta_M}^{t_M}}{t_M!} v^\lambda$, and define the rational (in fact, integer) $a_{v,w}$ by $\frac{e_{\beta_1}^{s_1}}{s_1!} \cdots \frac{e_{\beta_N}^{s_N}}{s_N!}\frac{f_{\beta_1}^{t_1}}{t_1!} \cdots \frac{f_{\beta_M}^{t_M}}{t_M!} v^\lambda=a_{v,w} v^{\lambda}$. The bilinear form $F(\cdot, \cdot)$ on this space defined by $F(v,w)=a_{v,w}$ is non-degenerate, and the following holds (\cite{jantzen}, II, 8.21).\\

   \begin{lemma}\label{lmrank}
   		Let $\lambda$ and $\mu$ be dominant weights with $\mu \prec \lambda$, and let $A$ be a matrix of the bilinear form on $V(\lambda)_\mu$ defined above, with respect to some basis of elements in $\mathscr{W}_{\lambda,\mu}$. Then the multiplicity of $\mu$ in $L(\lambda)$ is the number of elementary divisors of $A$ that are not divisible by $p$.
   \end{lemma}
   We will use this in the proof of the next lemma. To ease the notation, we denote $f_{i,j}=f_{\alpha_i+...+\alpha_j}$.\\
    
    \begin{lemma}\label{alvaro2222}
    
    If $\lambda=2\lambda_1+2\lambda_l$, and $\mu=0$, then $m_\lambda(\mu)=\binom{l+1}{2}-\epsilon_p(l+3)l-\epsilon_p(l+2)$.
    \end{lemma}
    \begin{proof}
    	Note first that we are considering $\lambda$ to be $p$-restricted and so $p \neq 2$. We start by finding a linearly independent set in $V(\lambda)_\mu$. For $1 \le j \le i \le l-1$, let $w_{i,j}=\frac{1}{ 2^{\delta_{i,j}} } f_{1,j}f_{1,i}f_{j+1,l}f_{i+1,l}v^{\lambda} $ and denote (in lexicographic order) $v_1=w_{1,1}$, $v_2=w_{2,1}$, $v_3=w_{2,2}$, $v_4=w_{3,1}$, ..., $v_{\binom{l}{2}}=w_{l-1,l-1}$. Also let $v_{\binom{l}{2}+k}=f_{1,k}f_{1,l}f_{k+1,l}v^{\lambda}$ for all $k$ such that $1 \le k \le l-1$ and $v_{\binom{l+1}{2}}= \frac{1}{2}f_{1,l}^2 v^{\lambda}$. Finally, let $\mathscr{S}=\{v_1,...,v_{\binom{l+1}{2}}\}$.
 By Freudenthal's formula, the multiplicity of $\mu$ in $V(\lambda)$ is $\binom{l+1}{2}$. Therefore proving that $\mathscr{S}$ is linearly independent will show that it is in fact a basis of $V(\lambda)_\mu$. We do this by computing the matrix $A$ of the bilinear form $F(\cdot, \cdot)$ defined above with respect to $\mathscr{S}$. For some positive integer $k \le l$, define first the $k \times l-1$ matrix $C_{k,l-1}$ as having 1s in the $(h,h)$ and $(h,k)$ entries for $1 \le h < k$, a 2 in the $(k,k)$ entry, and zeros elsewhere. Define also the $l-1 \times l-1$ matrix $C_{l-1,l-1}'$ as having 1s in its diagonal entries and the value $-1$ elsewhere. By use of the commutation relations in $\mathscr{L}$, one obtains the matrix $A$, which has the form
 
        \[
          A=
          \left(
          \renewcommand{\arraystretch}{1.08}
          \begin{array}{c|c|c}
          4I_{\binom{l}{2}} & \begin{array}{c}
          						-4C_{1,l-1} \\ 
                                \vdots \\
                                -4C_{l-1,l-1} \\
                               \end{array}  & 2J_{\binom{l}{2}}^T\\
                                
          \hline
          4C_{1,l-1}^T\ 4C_{2,l-1}^T\ \cdots\ 4C_{l-1,l-1}^T& 4C_{l-1,l-1}' & -6J_{l-1}^T\\
          \hline
          2J_{\binom{l}{2}} & 6J_{l-1} & 6
          \end{array}
          \right)
        \]
        where $I_k$ is the $k \times k$ identity matrix and $J_k$ is the $1 \times k$ vector of ones.
        Now define $P$ and $\tilde{P}$ as

	  \[
          P=
          \left(
          \renewcommand{\arraystretch}{1.1}
          \begin{array}{c|c}
          
          I_{\binom{l}{2}} & \,\makebox[6em]{\begin{tabular}{c|c}$C_{1,l-1}$&\\
                                $\vdots$ & $J_{\binom{l}{2}}$\\
                                $C_{l-1,l-1}$ &\end{tabular}}  \\
                                
          \hline
          \large0
         
          & \ C_{l,l}\ 
          \end{array}
          \right)
       \ \ \ \tilde{P}_{i,j} = \left\{
       \begin{array}{@{}l@{\thinspace}l}
        P_{i,i} &  \text{if } i=j\\
        P_{i,\binom{l+1}{2}} &  \text{if } i=\binom{l+1}{2},\ j\le \binom{l}{2}\\
        -P_{j,i} & \text{otherwise.}
       \end{array}
      \right.
        \]

An elementary check shows that $\tilde{P}AP = 4I_{\binom{l}{2}} \oplus 4(l+3)I_{l-1} \oplus (l+2)(l+3)I_1$. Hence $\mathscr{S}$ is a basis of $V(\lambda)_\mu$ and the result now follows from Lemma \ref{lmrank}.

    \end{proof}
    
	Now Table \ref{tableproof} contains each dominant weight $\lambda$ in Table \ref{tablefourth}, its subdominant weights, their multiplicities and the results we applied to obtain them. When the multiplicity is $1$ in the Weyl module, we do not cite any results. The dimensions in Table \ref{tablefourth} then follow from equations (\ref{dimbound}) and (\ref{orbitsize}). For the weight $\lambda_1+\lambda_2+\lambda_l$, the applications of Lemma \ref{lmseitz} are valid by the same argument in the proof of the Lemma (\cite{seitz}, 8.6). We do not consider the weights of the form $\lambda_k$ and $k\lambda_1$ as we discussed them in §\ref{dimsthird}. We also note that weights of the form $\lambda_1+\lambda_j$, $1<j<l$, only have one subdominant weight, $\lambda_{j+1}$. This has multiplicity $j-\epsilon_p(j+1)$ due to Lemma \ref{lmseitz}, so we omit this case too. 
 
\begin{table}
  \setlength{\tabcolsep}{0.5em}
  \renewcommand{\arraystretch}{1.2}
  \begin{tabular}{l l l c} 
  \specialrule{.1em}{.05em}{.05em} 
    $\lambda$ & $\mu$ & $m_\lambda(\mu)$ & Result used\\ 
    \specialrule{.1em}{.05em}{.05em}
    $2\lambda_2$&$\lambda_1+\lambda_3$ &$1$&\\ 

    & $\lambda_4$ & $2-\epsilon_p(3)$&  \ref{testerman121}\\ 
    \hline
    $\lambda=2\lambda_1+\lambda_2$& $2\lambda_2$ & $1$ &\\ 
    &$\lambda_1+\lambda_3$ & $2$ & \ref{lmseitz}\\ 
    &$\lambda_4$ & $3$&  \ref{cavallin2221}\\ 
   
    \hline
    $\lambda=3\lambda_1+\lambda_l$& $\lambda_1+\lambda_2+\lambda_l$, $\lambda_3+\lambda_{l}$ & $1$& \\ 
    &$2\lambda_1$ & $l-\epsilon_p(l+3)$& \ref{lmseitz}\\ 
    &$\lambda_2$ & $l-\epsilon_p(l+3)$ & \ref{cavallinc111}\\

    \hline
    $\lambda=2\lambda_1+\lambda_{l-1}$& $\lambda_2+\lambda_{l-1}$ & $1$ &\\

    & $\lambda_1+\lambda_{l}$ & $l-1-\epsilon_p(l+1)$& \ref{lmseitz}\\ 
    & $0$ & $\binom{l}{2}-\epsilon_p(l+1)(l-1)$& \ref{cavallin2221}\\ 
    \hline
    $\lambda_2+\lambda_{l-1}$ & $\lambda_1+\lambda_l$ & $l-2-\epsilon_p(l-1)$& \ref{lmseitz}   \\ 

    & $0$ & $\binom{l}{2}-1-\epsilon_p(l-1)l-\epsilon_p(l)$ &  \ref{cavallin1221}\\ 
     \hline
    $\lambda=2\lambda_1+2\lambda_l$& $\lambda_2+2\lambda_l$, $2\lambda_1+\lambda_{l-1}$, $\lambda_2+\lambda_{l-1}$ & $1$& \\ 
    
    & $\lambda_1+\lambda_l$ & $l-\epsilon_p(l+3)$ & \ref{lmseitz} \\ 
    & $0$ & $\binom{l+1}{2}-\epsilon_p(l+3)l-\epsilon_p(l+2)$ & \ref{alvaro2222}\\ 
    \hline
    $\lambda=\lambda_1+\lambda_2+\lambda_l$ & $\lambda_3+\lambda_{l}$ & $2-\epsilon_p(3)$&  \ref{lmseitz} \\
     
    &$2\lambda_1$ & $l-1-\epsilon_p(l)$&  \ref{lmseitz}\\
    & $\lambda_2$ & $2(l-1)-\epsilon_p(3)(l-2)-\epsilon_p(l)$\\ 
    & & $-\epsilon_p(l+2)+\epsilon_p(3)\epsilon_p(l+2)$ &   \ref{cavallin11001}\\
    
    \hline
    $\lambda=\lambda_2+\lambda_3$ & $\lambda_1+\lambda_4$ & $2-\epsilon_p(3)$ &  \ref{lmseitz}\\
    
    &$\lambda_5$ & $5-\epsilon_p(2)-4\epsilon_p(3)$ &   \ref{cavallin011}\\
    \hline
    $\lambda=3\lambda_1+\lambda_2$ &$\lambda_1+2\lambda_2$& $1$& \\ 

     & $2\lambda_1+\lambda_3$ & $2-\epsilon_p(5)$ & \ref{lmseitz}  \\
    &$\lambda_2+\lambda_3$ & $2-\epsilon_p(5)$ & \ref{cavallinc111}  \\
    &$\lambda_1+\lambda_4$ & $3-2\epsilon_p(5)$ &   \ref{cavallin2221} \\
    &$\lambda_5$ & $4-3\epsilon_p(5)$ &   \ref{cavallin3321}\\
   \specialrule{.1em}{.05em}{.05em} 
  \end{tabular} 
  \caption{Some weight multiplicities} \label{tableproof}
\end{table}

\subsection{Dominant weights}

Let $\lambda=a_1\lambda_1+...+a_l\lambda_l$ such that $\dim L(\lambda)\le (l+1)^4$ with $0 \le a_i < p$, and define $I_\lambda$ and $\Delta_\lambda$ as in §\ref{proofdom3}. We show that if $l>35$ then $\lambda$ appears in Table \ref{tablethird} or Table \ref{tablefourth}.

We again consider the different possibilities for $\Delta_\lambda$.

    Assuming $\Delta_\lambda \le l-5$, the minimum value $|\lambda^W|$ can attain occurs when $I_\lambda=\{6\}$, assuming $l>12$, in which case the orbit size exceeds $(l+1)^4$ for $l>32$.
    
    Assume now $\Delta_\lambda=l-4$. The minimum value $|\lambda^W|$ can attain then occurs when $I_\lambda=\{5\}$, assuming $l>10$, in which case the orbit size exceeds $(l+1)^4$ for $l>128$. This is attained by $\lambda_5$. If $\lambda = a_5\lambda_5$ with $a_5>1$, $\lambda-\alpha_5$ is subdominant and has nonzero $\lambda_6$ coefficient. By the previous paragraph we can discard this case. Alternatively if $\lambda \neq a_5\lambda_5$, the minimum orbit size is attained when $I_\lambda=\{1,5\}$, and it is greater than $(l+1)^4$ for $l>31$.
    
    We can now assume $\Delta_\lambda \ge l-3$. The following statements complete the proof.
    \begin{enumerate}[label=\alph*)]

	\item If $a_4>0$ and $a_1>0$, then $\lambda=\lambda_1+\lambda_4$. \label{fourth14}
    
	    \begin{proof} Note that $\mu=\lambda-(\alpha_1+\alpha_2+\alpha_3+\alpha_4)$ is subdominant with nonzero $\lambda_5$ coefficient. Setting $\mu=\lambda_5$ yields the stated $\lambda$.\end{proof}
        
    \item If $a_4>0$ and $a_1=0$, then $\lambda =\lambda_4$. \label{fourth4}
    
        \begin{proof} For the cases $a_4>1$, $a_3>0$, $a_2>0$, set respectively $\mu=\lambda-\alpha_4$, $\mu=\lambda-(\alpha_3+\alpha_4)$ and $\mu=\lambda-(\alpha_2+\alpha_3+\alpha_4)$. Then $\mu$ has nonzero $\lambda_5$ coefficient but $\mu \neq \lambda_5$, a contradiction.\end{proof}
        
        In the following we assume that $a_4=0$ and that $\lambda$ has either the form $a_1\lambda_1+a_2\lambda_2+a_3\lambda_3+a_l \lambda_l$ or $a_1\lambda_1+a_2\lambda_2+a_{l-1}\lambda_{l-1}+a_l \lambda_l$.
        
     \item If $a_3>0$ and $a_2>0$, then $\lambda=\lambda_2+\lambda_3$. \label{fourth23}

     \begin{proof} Clearly $\lambda-(\alpha_2+\alpha_3)$ is subdominant with nonzero $\lambda_1$ and $\lambda_4$ coefficients. Forcing it to equal $\lambda_1+\lambda_4$ yields the stated $\lambda$.\end{proof}
     
	\item If $a_3>0$ and $a_1>0$, then $\lambda=\lambda_1+\lambda_3$. \label{fourth13}
	
	\begin{proof} We have that $\mu=\lambda-(\alpha_1+\alpha_2+\alpha_3)$ is a subdominant weight with nonzero coefficient of $\lambda_4$, hence by cases \ref{fourth14} and \ref{fourth4} it is either $\lambda_4$ (which yields the stated $\lambda$) or $\lambda_1+\lambda_4$. Setting $\mu=\lambda_1+\lambda_4$ yields $\lambda=2\lambda_1+\lambda_3$. We may need to take this weight into account in subsequent cases, but we can discard it as highest weight as follows. Notice that its subdominant weights are $\lambda_2+\lambda_3$, $\lambda_1+\lambda_4$ and $\lambda_5$. The first has multiplicity 1 and the others are respectively greater than $1$ and $2$ due to Lemmas \ref{lmseitz} and \ref{cavallin2221}. These yield $\dim L(\lambda)>(l+1)^4$ for $l>35$.\end{proof}
	
     \item If $a_3>0$, $a_2=0$ and $a_1=0$, then $\lambda=\lambda_3$ or $\lambda=\lambda_3+\lambda_l$. \label{fourth3}
    
   \begin{proof} If $a_3>1$, $\lambda-\alpha_3$ has nonzero $\lambda_2$ and $\lambda_4$ coefficients, so we discard it by \ref{fourth14} and \ref{fourth4}. If $a_l>1$, $\lambda-\alpha_l$ has $\Delta_{\lambda-\alpha_l}= l-4$ and is not $\lambda_5$ or its dual, so we discard it too.\end{proof}
    
    From now we assume $a_3=0$, so $\lambda=a_1\lambda_1+a_2\lambda_2+a_{l-1}\lambda_{l-1}+a_l \lambda_l$.

        \item If $a_2>0$ and $a_{l-1}>0$, then $\lambda=\lambda_2+\lambda_{l-1}$. \label{fourth2l-1}
        
       \begin{proof}  We may only consider the two cases $a_2>1$ and $a_1>0$. Respectively, $\lambda-\alpha_2$ and $\lambda-(\alpha_1+\alpha_2)$ have nonzero $\lambda_3$ and $\lambda_{l-1}$ coefficients. In view of cases \ref{fourth23}, \ref{fourth13} and \ref{fourth3}, we can discard these.\end{proof}

	We can now assume that $\lambda$ has the form $a_1\lambda_1+a_2\lambda_2+a_l \lambda_l$.
    
    \item If $a_2>0$, $a_1>0$ and $a_{l}>0$, then $\lambda=\lambda_1+\lambda_2+\lambda_l$. \label{fourth12l}

	\begin{proof} Note that $\mu=\lambda-(\alpha_1+\alpha_2)$ is subdominant with nonzero $\lambda_3$ and $\lambda_l$ coefficients. By cases \ref{fourth23}, \ref{fourth13} and \ref{fourth3}, we see that $\mu=\lambda_3+\lambda_l$, yielding the stated $\lambda$.\end{proof}
    
    \item If $a_2>0$, $a_1>1$ and $a_{l}=0$, then $\lambda=3\lambda_1+\lambda_2$ or $\lambda=2\lambda_1+\lambda_2$. \label{fourth1112}
        
     \begin{proof}    The subdominant weight $\mu=\lambda-(2\alpha_1+2\alpha_2+\alpha_3)$ has nonzero $\lambda_4$ coefficient. By cases \ref{fourth14} and \ref{fourth4} we set $\mu=\lambda_4$ and $\mu=\lambda_1+\lambda_4$ and we obtain the stated values.\end{proof}

    \item If $a_2>0$, $a_1=1$ and $a_{l}=0$, then $\lambda=\lambda_1+\lambda_2$. \label{fourth12}
    
     \begin{proof}    Suppose $a_2>1$. Then $\mu=\lambda-(2\alpha_1+3\alpha_2+2\alpha_3+\alpha_4)$ is subdominant with nonzero $\lambda_5$ coefficient. Hence $\mu=\lambda_5$ which implies $\lambda=\lambda_1+2\lambda_2$. Computing the rank of the matrix of the form $F(\cdot, \cdot)$ on $V(\lambda)_\mu$ as in Lemma \ref{alvaro2222} (here the matrix is the same for all ranks) shows that $m_\lambda(\mu)= 5- \epsilon_p(3) \ge 4$. Using this, equation (\ref{dimbound}) yields $\dim L(\lambda) \ge (l+1)^4$ for $l>33$. \end{proof}

     \item If $a_2>0$, $a_1=0$ and $a_{l}>0$, then $\lambda=\lambda_2+\lambda_l$ or $\lambda=\lambda_2+2\lambda_l$. \label{fourth2l} 
     
    \begin{proof}  If $a_2>1$, the subdominant weight $\mu=\lambda-(\alpha_1+2\alpha_2+\alpha_3)$ has nonzero $\lambda_4$ and $\lambda_l$ coefficients. In view of \ref{fourth14} and \ref{fourth4}, we discard this case. Now if $a_l>2$, $\mu=\lambda-(\alpha_{l-1}+2\alpha_l)$ is subdominant with $\Delta_\mu=l-4$ and nonzero $\lambda_2$ and $\lambda_{l-2}$ coefficients, so we discard this too.\end{proof}

     \item If $a_2>0$, $a_1=0$ and $a_{l}=0$, then $\lambda=\lambda_2$ or $\lambda=2\lambda_2$. \label{fourth22}
    
     \begin{proof} Otherwise if $a_2>2$, $\lambda-(\alpha_1+2\alpha_2+\alpha_3)$ is subdominant with nonzero $\lambda_4$ and $\lambda_2$ coefficients, so not as in \ref{fourth14} or \ref{fourth4}.\end{proof}
     
     Finally, we assume $\lambda$ has the form $a_1 \lambda_1+a_l \lambda_l$ and $a_1 \ge a_l$.

    \item If $a_1>2$, then $\lambda=3\lambda_1$, $\lambda=4\lambda_1$, $\lambda=5\lambda_1$ or $\lambda=3\lambda_1+\lambda_l$. \label{fourth1111}
	
    \begin{proof} If $a_l>0$, $\mu=\lambda-2\alpha_1-\alpha_2$ has nonzero $\lambda_3$ coefficient and by cases \ref{fourth23}, \ref{fourth13} and \ref{fourth3} we must have $\mu=\lambda_3+\lambda_l$, which yields $\lambda=3\lambda_1+\lambda_l$. Assuming $a_l=0$, $\dim L(a_1\lambda_1)=\binom{l+a_1}{a_1}$. If $a_1>5$, this exceeds $(l+1)^4$ for $l>17$.\end{proof}
	
	\end{enumerate}
The cases not considered at this point belong to weights already in Tables \ref{tablethird} or \ref{tablefourth} and so the proof of Theorem \ref{thmfourth} is complete.

\section{Explicit descriptions of the modules in Theorem \ref{thmthird}} \label{descriptions}
In this section we find constructions of the modules corresponding to the weights in Table \ref{tablethird}. We omit the weights $\lambda_k$, $k\lambda_1$ and $\lambda_1+\lambda_l$ as the respective constructions were described in §\ref{dimsthird}.

To start, the Young symmetrizers construction (or Weyl's construction, see \cite{fultonharris}, Lecture 6) gives the Weyl module as a subspace of the tensor product $V^{\otimes k}$, where $V=V(\lambda_1)$, the natural module. We denote by $\{e_1, ..., e_{l+1}\}$ the canonical basis of $V$. In order to describe $L(\lambda)$, we find the composition factors of the Weyl module, which, in view of §\ref{dimsthird}, in these cases consist of $L(\lambda)$ and (possibly) another module with multiplicity one and with highest weight $\lambda-(\alpha_i+...+\alpha_j)$.

We use the same notation as in §\ref{dimsfourth} for the Lie algebra $\mathscr{L}$ and its elements. The following result is an immediate consequence of 4.1.2 from \cite{cavallin}.\\
     
     \begin{lemma} \label{lindep}
     With the notation of Lemma \ref{lmseitz}, if $p \mid a_i+a_j+j-i$, then the weight $\mu$ affords the highest weight of a $KG$-composition factor of the Weyl module $V(\lambda)$, with highest weight vector $a_j f_{i,j}v^{\lambda}-\sum_{r=i+1}^{j}f_{i,r-1}f_{r,j}v^{\lambda}$, where $v^{\lambda}$ is a highest weight vector of $V(\lambda)$.
     \end{lemma}
     In the following, we use this result to give explicit constructions of the modules.
	\begin{description}
	\item{Construction for $\lambda=\lambda_1+\lambda_2$}
    	
        The Weyl module has the following form (\cite{fultonharris}, Lecture 6):
        \[V(\lambda)=Span\{v_1 \otimes v_2 \otimes v_3+v_2 \otimes v_1 \otimes v_3-v_3 \otimes v_2 \otimes v_1-v_3 \otimes v_1 \otimes v_2 : v_1, v_2, v_3 \in  V\}. \]
        
        As noted in (\cite{fultonharris}, p. 76), if we identify $V \otimes \Lambda^2 V$ as subspace of $V^{\otimes 3}$, then $V(\lambda)$ can also be realised as $\mathrm{Ker}(V \otimes \Lambda^2 V \rightarrow \Lambda^3 V)$, the kernel of the canonical map from $V \otimes \Lambda^2 V$ to $\Lambda^3 V$.
        
        Now the highest weight vector of $V(\lambda)$ can be taken as $v^{\lambda}=e_1 \otimes (e_1 \wedge e_2)$. The only subdominant weight is $\lambda_3$. By Lemma \ref{lindep}, $\lambda_3$ affords the highest weight of a composition factor of $V(\lambda)$ as a $KG$-module precisely when $p=3$. In this case the weight space $V(\lambda)_{\lambda_3}$ is spanned by the vector 
        
        \[ \begin{array} {ll}
        v_R &:= f_{1,2}v^{\lambda}-f_{1,1}f_{2,2}v^{\lambda}\\
        &=2 e_1 \otimes (e_3 \wedge e_2)-e_2 \otimes (e_1 \wedge e_3)+e_3 \otimes (e_1 \wedge e_2)\\
        &= e_1 \otimes (e_2 \wedge e_3)-e_2 \otimes (e_1 \wedge e_3)+e_3 \otimes (e_1 \wedge e_2)\\
        &=e_1 \wedge e_2 \wedge e_3.
        \end{array} \]
        
        Note that, as expected, $R=\mathscr{L} v_R$ is a submodule isomorphic to $\Lambda^3 V$ (since in fact, $R=\Lambda^3 V$). We conclude that
        \[
  L(\lambda_1+\lambda_2) = \left\{
     \begin{array}{@{}l@{\thinspace}l}
       \mathrm{Ker}(V \otimes \Lambda^2 V \rightarrow \Lambda^3 V)  &  \text{ if }p \neq 3\\
      \mathrm{Ker}(V \otimes \Lambda^2 V \rightarrow \Lambda^3 V)/ \Lambda^3 V  &  \text{ if }p=3.\\
     \end{array}
   \right.
	    \]

     \item{Construction for $\lambda=\lambda_1+\lambda_{l-1}$}
    	
        The Weyl module is in this case spanned by the vectors in the kernel of the canonical map $V \otimes \Lambda^{l-1} V \rightarrow \Lambda^l V$. The highest weight vector can be taken as $v^{\lambda}=e_1 \otimes (e_1 \wedge e_2 \wedge ... \wedge e_{l-1})$. The only subdominant weight is $\lambda_l$. In the case $p\mid l$, in view of Lemma \ref{lindep} we define
        
        \[v_R = f_{1,l-1}v^{\lambda}-\sum_{r=2}^{l-1}f_{1,r}f_{r,l-1}v^{\lambda}=\pm e_1 \wedge ... \wedge e_l. \]
        
        Now $R=\mathscr{L} v_R = \Lambda^l V $ is a submodule of $V(\lambda)$ and
        \[
  L(\lambda_1+\lambda_{l-1}) \cong \left\{
     \begin{array}{@{}l@{\thinspace}l}
      \mathrm{Ker}(V \otimes \Lambda^{l-1} V \rightarrow \Lambda^l V)   &  \text{ if }p \nmid l\\
      \mathrm{Ker}(V \otimes \Lambda^{l-1} V \rightarrow \Lambda^l V)/ \Lambda^l V &  \text{ if } p\mid l.\\
     \end{array}
   \right.
	    \]
        \item{Construction for $\lambda=2\lambda_1+\lambda_{l}$}
    	
        The Weyl module is again the image in $V^{\otimes l+2}$ of the corresponding Young symmetrizer, which in turn corresponds to the tableau associated to the partition $(3,1,...,1)$. A highest weight vector is $(e_1 \cdot e_1) \otimes (e_1 \wedge ... \wedge e_{l}) $.  To see that this lies in $V(\lambda)$, notice that it is the image of the vector $e_{1} \otimes e_{1} \otimes e_{1} \otimes e_{2} \otimes ... \otimes e_{l}$  under the Young symmetrizer. Now, the only subdominant weight that could afford the highest weight of a composition factor is $\lambda_1$. By Lemma \ref{lmseitz} this happens precisely when $p \mid l+2$ and, in that case, by Lemma \ref{lindep}, we define
        
        \[v_R = f_{1,l}v^{\lambda}-\sum_{r=2}^{l}f_{1,r-1}f_{r,l}v^{\lambda}=\pm 2\left (e_1 \otimes (e_1 \wedge ... \wedge e_{l+1})+\varphi(e_1 \otimes (e_1 \wedge ... \wedge e_{l+1}))\right),\]
		
        where $\varphi : V^{\otimes l+2} \rightarrow V^{\otimes l+2}$ is the linear map that swaps the first two entries of the basis tensors:
        $\varphi(e_{b_1} \otimes e_{b_2} \otimes e_{b_3} \otimes ... \otimes e_{b_{l+2}})=e_{b_2} \otimes e_{b_1} \otimes e_{b_3} \otimes ... \otimes e_{b_{l+2}}$, for $b_i \in \{1,2,...,l+1\}$.
        
        Finally,
        \[
  L(2\lambda_1+\lambda_{l})= \left \{
     \begin{array}{@{}l@{\thinspace}l}
      V(\lambda)   &  \text{ if }p \nmid l+2\\
      V(\lambda)/ R &  \text{ if } p\mid l+2.\\
     \end{array}
     \right.
	    \]
        
        where $R=\mathscr{L} v_R = \{(id + \varphi)(e_i \otimes (e_1 \wedge ... \wedge e_{l+1})): 1 \le i \le l+1  \} $, and as expected, $R \cong V$. 
\end{description}
\section*{Acknowledgements}
I wish to thank my supervisor Martin W. Liebeck for his encouragement and priceless advice throughout the work in this paper. I am also grateful for the financial support of the Imperial College UROP Award.

\setlength{\bibsep}{2pt}

\'Alvaro L. Mart\'inez, \emph{Imperial College London}, London SW7 2AZ, UK. \url{alm116@ic.ac.uk}

\end{document}